\documentclass[11pt]{amsart}

\usepackage{a4}

\usepackage{latexsym}
\usepackage{amssymb}
\usepackage{amsmath}
\usepackage{amscd}
\usepackage{mathrsfs}

\newtheorem{theo}{Theorem}[section]
\newtheorem{lemma}[theo]{Lemma}
\newtheorem{ex}[theo]{Example}
\newtheorem{prop}[theo]{Proposition}
\newtheorem{cor}[theo]{Corollary}

\theoremstyle{definition}
\newtheorem{rem}[theo]{Remark}

\DeclareMathOperator{\Span}{{\rm span}}

\DeclareMathOperator{\supp}{{\rm supp}} 
\DeclareMathOperator{\Stab}{{\rm Stab}} 
\DeclareMathOperator{\Ad}{{\rm Ad}} 
\DeclareMathOperator{\vol}{{\rm vol}} 
\DeclareMathOperator{\Cc}{{\rm C}_{{\rm c}}} 

\newcommand{\vect}[1]{{\boldsymbol{#1}}}
\newcommand{\vb}{\vect{b}}
\newcommand{\vh}{\vect{h}}

\newcommand{\vp}{\vect{p}}
\newcommand{\vr}{\vect{r}}

\newcommand{\vs}{\vect{s}}

\newcommand{\vu}{\vect{u}}
\newcommand{\vv}{\vect{v}}
\newcommand{\vx}{\vect{x}}

\newcommand{\vz}{\vect{z}}

\newcommand{\field}[1]{\mathbb{#1}} 
 
\newcommand{\R}{\field{R}}
\newcommand{\N}{\field{N}}
\newcommand{\Q}{\field{Q}}
\newcommand{\Z}{\field{Z}}

\providecommand{\abs}[1]{\lvert#1\rvert}

\providecommand{\norm}[1]{\lVert#1\rVert}

\providecommand{\trn}[1]{{\,{}^{\bf t}\!#1}}
\providecommand{\card}[1]{\#(#1)}
\providecommand{\ignore}[1]{}

\renewcommand{\setminus}{\smallsetminus}

\newcommand{\inv}{^{-1}}

\newcommand{\la}[1]{\mathfrak{\lowercase{#1}}}

\newcommand{\cK}{\mathcal{K}}

\newcommand{\cU}{\mathcal{U}}

\newcommand{\cV}{\mathcal{V}}

\begin{document}

\title[Khinchin theorem for quadratic varieties]{Khinchin theorem for
  integral points on quadratic varieties} 
\author{Alexander Gorodnik and Nimish A. Shah} 
\address{School of Mathematics \\ University of
  Bristol \\ Bristol BS8 1TW, UK} 
\email{a.gorodnik@bristol.ac.uk}
\address{Tata Institute of Fundamental Research, Mumbai 400005, INDIA}
\email{nimish@math.tifr.res.in}

\begin{abstract}
  We prove an analogue the Khinchin theorem for the Diophantine
  approximation by integer vectors lying on a quadratic variety. The
  proof is based on the study of a dynamical system on a homogeneous
  space of the orthogonal group. We show that in this system, generic
  trajectories visit a family of shrinking subsets infinitely often.
\end{abstract}



\maketitle

\section{Introduction}

Let us consider the following question in the theory of Diophantine
approximation: given a vector $v\in\mathbb{R}^{d}$, is it possible to
approximate $v$ by a sequence of rational vectors $\frac{x}{y}$ that
come from integer points lying on the quadratic variety
$x_1^2\pm\cdots\pm x_d^2-y^2=1$?  Namely, we are interested in the
integral solutions $(x,y)\in\mathbb{Z}^{d+1}$ of
\begin{equation}\label{eq:dioph0}
  \left\|v-\frac{x}{y}\right\|<\epsilon,
  \quad\quad x_1^2\pm\cdots\pm x_d^2-y^2=1. 
\end{equation}
It is easy to see that if (\ref{eq:dioph0}) has infinitely many
integral solutions for every $\epsilon>0$, then the vector $v$ has to
lie on the variety $\mathcal{Q}=\{v_1^2\pm\cdots\pm v_d^2=1\}$.  On
the other hand, it follows, for instance, from the results in
\cite{G+O+S:Satake} (see \cite[Corollary~1.7, \S2.1]{G+O+S:Satake})
that for every $v\in \mathcal{Q}$ and $\epsilon>0$, (\ref{eq:dioph0})
has infinitely many solutions.  In this paper, we consider a more
delicate question about the order of approximation in
(\ref{eq:dioph0}).  For a given function $\psi:(0,\infty)\to
(0,\infty)$, we study whether there are infinitely many integral
solutions $(x,y)\in\mathbb{Z}^{d+1}$ of
\begin{equation}\label{eq:dioph2}
  \left\|v-\frac{x}{y}\right\|<\psi(|y|),
  \quad\quad x_1^2\pm\cdots\pm x_d^2-y^2=1;
\end{equation}
and show that the answer is determined by integrability of the
function $t^{d-2}\psi(t)^{d-1}$.  This result is analogous to the
Khinchin theorem, which we now recall. A vector $v\in \mathbb{R}^{d}$
is called {\it $\psi$-approximable} if the inequality
$$
\left\|v-\frac{x}{y}\right\|<\psi(|y|)
$$
has infinitely many integral solutions $(x,y)\in
\mathbb{Z}^{d+1}$. The Khinchin theorem determines the size of the set
of $\psi$-approximable vectors:

\begin{theo}[Khinchin]
  Let $\psi:(0,\infty)\to (0,\infty)$ be a non-increasing
  function. Then the following statements hold:
  \begin{enumerate}
  \item[(i)] If $\int_1^\infty t^d \psi(t)^{d}dt=\infty$, then almost
    every vector in $\mathbb{R}^d$ is $\psi$-approximable.
  \item[(ii)] If $\int_1^\infty t^d\psi(t)^{d}dt<\infty$, then almost
    every vector in $\mathbb{R}^d$ is not $\psi$-approximable.
  \end{enumerate}
\end{theo}

A function $\psi: (0,\infty)\to (0,\infty)$ is called {\it
  quasi-conformal} if there exists $c>0$ such that
$$
\hbox{$\psi(ht)\leq c\,\psi(t)$ for all $h\in[1/2,2]$ and $t>0$.}
$$
It was proved by Sullivan \cite[\S3]{Sul} that the Khinchin theorem
also holds for measurable quasi-conformal functions (which are not
necessarily non-increasing).

Note that question (\ref{eq:dioph2}) is about approximation by radial
projections of integral points lying on the variety $x_1^2\pm\cdots\pm
x_d^2-y^2=1$ to the plane $\{y=1\}$, and when $\psi$ is
quasi-conformal, it does not depend on a choice of the radial
projection.  Hence, it is natural to restate question (2) in a more
general geometric fashion.

Let $X$ be an algebraic variety in the Euclidean space
$\mathbb{R}^{d+1}$.  We denote by $\pi:\mathbb{R}^{d+1}\backslash
\{(0,\ldots, 0)\}\to S^{d}$ the radial projection on the unit sphere
$S^d$.  We say that a vector $v\in S^d$ is {\it
  $(X,\psi)$-approximable} if the inequality
$$
\|\pi(x)-v\|<\psi(\|x\|).
$$
has infinitely many solutions $x\in
X(\mathbb{Z}):=X\cap\mathbb{Z}^{d+1}$.

Now for quasi-conformal functions $\psi$, problem (\ref{eq:dioph2})
can be restated as a question about $(X,\psi)$-approximable vectors,
where $X=\{x_1^2\pm\cdots\pm x_d^2-y^2=1\}$, and the Khinchin theorem
is about $(\mathbb{R}^{d+1},\psi)$-approximable vectors.

We define the {\it boundary} $\partial X$ of a variety $X$ to consist
of the points $\lim_{n\to\infty} \pi(x_n)$ with $x_n\in
X\backslash\{(0,\ldots,0)\}$, $\|x_n\|\to \infty$.  Note that if
$\psi(t)\to 0$ as $t\to\infty$, then the set of
$(X,\psi)$-approximable vectors is always contained in $\partial X$.

\subsection{Quadratic varieties} \label{sec:quad} Let $X$ be a
nonsingular rational quadratic; that is, $X=\{w\in\R^{d+1}:Q(w)=m\}$
for some $m\in \mathbb{Q}\setminus\{0\}$, where $Q$ is a rational
nondegenerate indefinite quadratic form.  In this case,
\[
\partial X=\{x\in\R^{d+1}:Q(x)=0\}\cap S^{d}.
\]

We assume that $d\ge 3$ and $X(\Z)\neq\emptyset$.

Let $G=\hbox{O}(Q)$ be the orthogonal group. By Witt's theorem $G$
acts transitively on $X$.  The variety $X$ supports a $G$-invariant
measure, which we denote by $\vol$.  We also consider the smooth
action of $G$ on $S^{d}$ by $g\cdot \pi(w)=\pi(gw)$ for all $g\in G$
and $w\in \R^{d+1}$. Under this action, $\partial X$ is homogeneous
space of $G$ admitting a unique $G$-semi-invariant probability measure
$\mu_\infty$.

\begin{theo} \label{thm:main-qform} Let the notation be as above and
  $\psi:(0,\infty)\to (0,\infty)$ a measurable quasi-conformal
  function. Then the following statements hold:
  \begin{enumerate}
  \item[(i)] If $\int_1^\infty t^{d-2}\psi(t)^{d-1}\,dt=\infty$, then
    $\mu_\infty$-almost every $v\in \partial X$ is
    $(X,\psi)$-approximable.
  \item[(ii)] If $\int_1^\infty t^{d-2}\psi(t)^{d-1}\,dt<\infty$, then
    $\mu_\infty$-almost every $v\in \partial X$ is not
    $(X,\psi)$-approximable.
  \end{enumerate}
\end{theo}

\begin{rem}
  \begin{enumerate}
  \item It is important to note that Theorem \ref{thm:main-qform}
    holds only almost everywhere. There are examples of not
    $(X,\psi)$-approximable vectors under assumption (i) and examples
    of $(X,\psi)$-approximable vectors under assumption (ii) (see
    Section \ref{sec:examples}).

  \item Since the function $\psi$ is quasi-conformal, it is clear that
    the claim of Theorem \ref{thm:main-qform} does not depend on a
    choice of the norm. Hence, we assume that $\|\cdot\|$ is the
    standard Euclidean norm.

  \item In this paper, we consider the variety $\{Q=m\}$ with $m\ne
    0$.  The structure of the singular variety $\{Q=0\}$ is
    different. It seems that the integral/rational points on the later
    variety can studied using the methods from \cite{drutu}.
  \end{enumerate}
\end{rem}

For $v\in \partial X$, we define the {\it cusp} $C(v,\psi)$ at $v$:
\begin{equation}
  \label{eq:23}
  C(v,\psi)=\{x\in X:\, \|\pi(x)-v\|<\psi(\|x\|)\}.  
\end{equation}
Using this notation, Theorem \ref{thm:main-qform} can be restated as
follows:

\begin{theo}\label{th:2}
  For $\mu_\infty$-almost every $v\in \partial X$,
  \[\vol(C(v,\psi))=\infty\quad\Longleftrightarrow\quad
  \#(X(\mathbb{Z})\cap C(v,\psi))=\infty.
  \]
\end{theo}

It follows from Theorem \ref{thm:vol-est} below that Theorem
\ref{th:2} is equivalent to Theorem \ref{thm:main-qform}, and the
condition $\vol(C(v,\psi))=\infty$ is independent of $v\in\partial X$.

\ignore{ In view of this, Theorem~\ref{thm:main-qform} is a
  consequence of the following:

  \begin{theo}\label{thm:cusps}
    \begin{enumerate}
    \item If $\vol(C(v,\psi))=\infty$ for some $v\in\partial X$ then
      $C(v,\psi)\cap X(\mathbb{Z})$ is infinite for a.e.\
      $v\in\partial X$.
    \item If $\vol(C(v,\psi))<\infty$ for some $v\in\partial X$, then
      $C(v,\psi)\cap X(\mathbb{Z})$ is finite for a.e.\ $v\in\partial
      X$.
    \end{enumerate}
  \end{theo}

  We will also obtain the following analogue of
  Theorem~\ref{thm:main-qform}.

  \begin{cor} \label{cor:dio-soln} Let $n\geq 4$ and let $Q_0$ be a
    nondegenerate non-negative definite quadratic form in $n-1$
    variables and with integral coefficients. Suppose that there exist
    $\vx_0\in Z^{n-1}$ and $y_0\in\Z$ such that
    $Q_0(\vx_0)-y_0^2=:m\neq 0$. Let $\psi:[1,\infty)\to\infty$ be a
    quasi-conformal map. Then for almost any $v$ in the variety
    $X_0:=\{v\in R^{n-1}:Q_0(v)=1\}$, with respect to the
    $\hbox{O}(Q_0)$-invariant measure on the variety, we have
    \begin{equation} \label{eq:soln} \# \{(\vx, y)\in \Z^{n-1}\times
      \Z: \norm{\frac{\vx}{y}- v}<\psi(\abs{y}),\;
      Q_0(\vx)-y^2=m\}=\infty
    \end{equation}
    if and only if
    \begin{equation}
      \label{eq:45}
      \int_0^\infty(e^t\psi(e^t))^{n-1}\,dt=\infty.    
    \end{equation}
  \end{cor}

  In fact, the above study was motivated by the following result,
  which is a specialization of \cite[Cor.1.7,\S2.1]{G+O+S:Satake} to
  the quadratic variety case.

  \begin{cor} \label{cor:dio-count} Let the notation be as in
    Corollary~\ref{cor:dio-soln}. Then Then given $\epsilon>0$, there
    exists a computable constant $c=c(\epsilon)>0$ such that for every
    $v\in\R^{n-1}$ with $Q_0(v)=1$,
    \begin{equation} \label{eq:count} \# \{(\vx, y)\in \Z^{n-1}\times
      \Z: \norm{\frac{\vx}{y}- v}<\epsilon,\;
      Q_0(\vx)-y^2=m,\;\norm{(\vx,y)}\leq T\}\sim c\cdot T^{n-2}.
    \end{equation}
  \end{cor}

  This counting result corresponds to the case of
  $\psi(t)=\epsilon$. It is not clear whether one can obtain such
  exact counting estimate if $\psi(t)\to 0$ as $t\to\infty$ for every
  $v\in X_0$. It is important to note that for general $\psi$ the
  above results hold only almost every where. There are examples of
  quadratic forms where \eqref{eq:soln} holds for some $v\in\R^{n-1}$
  even though \eqref{eq:45} fails. Conversely, there exists $\psi$
  such that \eqref{eq:45} holds but \eqref{eq:soln} fails to hold for
  some $v\in X_0$.  }

\subsection{Shrinking targets} 
We will prove Theorem~\ref{thm:main-qform} using a shrinking target
property for a flow on a suitable homogeneous space.  Let $G$ be a
noncompact real algebraic group and $H$ an algebraic subgroup which is
the set of fixed points of an involution $\sigma$.  We fix a
nontrivial one parameter subgroup $\{a_t\}$ of $G$ such that
$\sigma(a_t)=a_{-t}$.  Let $Z$ denote the centralizer of $\{a_t\}$ in
$G$ and $U^-$ the contracting horospherical subgroup of $a_t$, i.e.,
$$
U^-=\{g\in G:\, a_tga_t^{-1}\to e\;\hbox{as $t\to\infty$}\}.
$$

Given a lattice $\Gamma$ in $G$, we consider the flow $a_t$ on the
space $G/\Gamma$. We are interested in visits of generic trajectories
of $a_t$ to shrinking boxes of the form $\Psi_tB\Gamma\subset
G/\Gamma$ where $\Psi_t\subset U^-$ and $B\subset ZH$. The following
is our main result:

\begin{theo}
  \label{thm:bcl}
  Given a bounded measurable subset of $B\subset ZH$ of positive
  measure, there exists a neighborhood $\mathcal{O}$ of identity in
  $U^-$ such that for any measurable subsets $\Psi_n\subset
  \mathcal{O}$, an increasing sequence $t_n\to\infty$ of real numbers,
  and $y_0\in G/\Gamma$, the following statements hold:
  \begin{enumerate}

  \item[(i)] If $\inf_{n\in\N} (t_{n+1}-t_n)>0$ and $\sum_{n=1}^\infty
    \vol_{U^-}(\Psi_{n})=\infty$, then
    \[
    \vol_{G/\Gamma}(\{z\in G/\Gamma:\card{\{n\in \N:\, a_{t_n}^{-1}
      z\in \Psi_{n}By_0\}}=\infty\})>0.
    \]

  \item[(ii)] If $\sum_{n=1}^\infty \vol_{U^-}(\Psi_{n})<\infty$, then
    \[
    \vol_{G/\Gamma}(\{z\in G/\Gamma:\card{\{n\in\N:\, a_{t_n}^{-1}
      z\in \Psi_{n}B y_0\}}=\infty\})=0.
    \]
  \end{enumerate}
\end{theo}

The problem of shrinking targets, that is, the problem about visits of
trajectories to a family shrinking subsets, has been an active topic
of research over the past decades (see
\cite{bv,ck,Do,Ga,hp1,hp2,KM,Mau,Sul}). It seems that in the context
of partially hyperbolic systems, there are two main approaches to this
question.  One is based on (strong) mixing properties of the flow, and
the other uses geometric properties of the space such as negative
curvature.  As usual, the crucial step is to show that the sets
$\{a_{t_n}\Psi_{n}B\Gamma\}$ are quasi-independent (see Proposition
\ref{prop:dynbcl} and Theorem \ref{thm:vol-B_t}).  Our proof of
quasi-independence is quite different from the previous works and is
based on the simple observation that the sets $a_{t_n}(\Psi_{n}B)
a_{t_n}^{-1}$ and $a_{t_m}(\Psi_{m}B) a_{t_m}^{-1}$ are
``transversal'' if $|n-m|$ is sufficiently large (see the proof of
Proposition \ref{thm:inter-vol}).

As a consequence of Theorem \ref{thm:bcl}, we obtain

\begin{cor}
  \label{cor:bcl} Suppose that $\Psi_{n+1}\subset\Psi_{n}$ for all
  $n\in\N$, $\sum_{n=1}^\infty \vol_G(\Psi_{n})=\infty$, and the
  action of $T=a_{t_0}$ on $G/\Gamma$ is ergodic for some $t_0>0$.
  Then for almost every $z \in G/\Gamma$,
$$
\card{\{n\in \N:\, T^{-n} z\in \Psi_{n}By_0\}}=\infty.
$$
\end{cor}

\subsection{Organization of the paper}
In Section \ref{sec:vol}, we introduce a convenient coordinate system
to describe the cusps $C(v,\psi)$ and show that Theorems
\ref{thm:main-qform} and \ref{th:2} are equivalent. Section
\ref{sec:inter} contains the proof of quasi-independence which is
crucial for Theorem \ref{thm:bcl}.  In Section \ref{sec:proof}, we
prove the theorems from the introduction, and in Section
\ref{sec:examples}, we give some examples for Theorem
\ref{thm:main-qform}.

\subsection{Acknowledgments}{\small We would like to thank Anish Ghosh
  for useful discussions.  N.S. would like to thank Caltech for
  hospitality where this work was started.  A.G. would like to thank
  Princeton University for hospitality during the academic year
  2007--2008 where part of this project was completed.  A.G. was
  supported by NSF grant 0654413 and RCUK Fellowship.}

\section{Description of the cusps}\label{sec:vol}

In this section, we use notation from Section \ref{sec:quad}, and we
will prove

\begin{theo}
  Let $\psi:(0,\infty)\to (0,\infty)$ be a measurable quasi-conformal
  function. Then for any $v\in \partial X$,
  \label{thm:vol-est}
  \[\vol(C(v,\psi))=\infty\quad\iff\quad
  \int_1^\infty t^{d-2}\psi(t)^{d-1}\,dt=\int_0^\infty
  (e^t\psi(e^t))^{d-1}\, dt=\infty.
  \]
\end{theo}

Recall that we are assuming that $\|\cdot\|$ is the standard Euclidean
norm. In this case, $\pi(x)=x/\norm{x}$.

We choose a basis $\{f_1,\dots,f_{d+1}\}$ of $\R^{d+1}$ such that
\[
Q(x_1f_1+\dots+x_{d+1}f_{d+1})=2x_1x_{d+1}+x_2^2+\cdots+x_p^2-x_{p+1}^2-\cdots-x_{d}^2,
\quad (x_i\in\R).
\]
As we have noted before, the variety $X=\{Q=m\}$, where
$m\in\Q\setminus\{0\}$, is a homogeneous space of $G=\hbox{O}(Q)\simeq
\hbox{O}(p,q)$.

\begin{lemma} \label{lema:qc} Given a compact set $\cK\subset G$,
  there exists a constant $\kappa>1$ such that for any $v\in \partial
  X$ and $g\in \cK$,
  \[
  g(C(v,\psi))\subset C(\pi(gv),\kappa\psi).
  \]
\end{lemma}

\begin{proof}
  Let $x\in C(v,\psi)$. Then
$$
\|\pi(x)-v\|<\psi(\|x\|).
$$
We need to prove that $gx\in C(\pi(gv),\kappa \psi)$ for some
$\kappa>0$, i.e.,
\begin{equation}\label{eq:cone}
  \|\pi(gx)-\pi(gv)\|<\kappa\psi(\|gx\|).
\end{equation}
We have
\begin{align*}
  \|\pi(gx)-\pi(gv)\|=\left\|\frac{g\pi(x)}{\|g\pi(x)\|}-\frac{gv}{\|gv\|}\right\|,
\end{align*}
and using the inequality
$$
\left\|\frac{w_1}{\|w_1\|}-\frac{w_2}{\|w_2\|}\right\|\le
\frac{2}{\|w_2\|}\|w_1-w_2\|,
$$
we deduce that for some $\kappa_1=\kappa_1(\mathcal{K})>0$,
\begin{align*}
  \|\pi(gx)-\pi(gv)\| &\le \frac{2}{\|gv\|}\|g\pi(x)-gv \|\\
  &\le \frac{2\|g\|}{\|gv\|}\|\pi(x)-v \| <\kappa_1\psi(\|x\|).
\end{align*}
Since $\psi$ is quasi-conformal, there exists
$\kappa_2=\kappa_2(\mathcal{K})>0$ such that
$$
\psi(\|x\|)\le \kappa_2 \psi(\|gx\|).
$$
This implies (\ref{eq:cone}).
\end{proof}

\begin{lemma} \label{lema:phi} Let $\phi:\R\to (0,\infty)$ be a
  measurable function such that for some $c>1$,
  \begin{equation}\label{eq:36}
    c\inv\phi(x)\leq \phi(x+h)\leq c \phi(x), 
    \quad  x\in\R,\ h\in[-1,1]. 
  \end{equation}

  \begin{enumerate}
  \item[(i)] If $\int_0^\infty \phi(t)\,dt<\infty$, then $\phi(t)\to
    0$ as $t\to\infty$.

  \item[(ii)] If $\int_0^\infty\phi(t)\,dt=\infty$, then there exists
    a function $\phi_1:(0,\infty)\to (0,1]$ such that $\phi_1$
    satisfies (\ref{eq:36}) with a possibly different constant $c>0$,
    $\phi_1\le \phi$, $\phi_1(t)\to 0$ as $t\to\infty$,
    $\frac{\phi(t)}{\phi_1(t)}\to \infty$ as $t\to\infty$, and
    $\int_0^\infty\phi_1(t)\,dt=\infty$.
  \end{enumerate}
\end{lemma}

We note that the function $\phi(t):=e^t\psi(e^t)$ satisfies the
condition \eqref{eq:36} for some $c>1$.

\begin{proof}
  Suppose there exists $\delta>0$ and a sequence $t_i\to\infty$ such
  that $\phi(t_i)\geq\delta$. We may assume that $t_{i+1}-t_i\geq
  2$. By \eqref{eq:36}, $\phi(t)\geq c^{-1}\delta$ for all
  $t\in[t_i-1,t_i+1]$. Therefore $\int_0^\infty
  \phi(t)\,dt=\infty$. In particular, this proves (i).

  Now we prove (ii). Let $\phi_2(t)=\min(\phi(t),1)$.  Then it follows
  from (\ref{eq:36}) that $\int_0^\infty \phi_2(t)\,dt=\infty$, and
  $\phi_2$ satisfies (\ref{eq:36}) with a different constant $c>0$.
  Let $T_0=0$ and $T_i>0$ be such that $\int_0^{T_i}\phi_2(t)\,dt=i$
  (it follows from (\ref{eq:36}) that the function $T\mapsto \int_0^T
  \phi_2(t)dt$ is continuous).  Then $T_i-T_{i-1}\geq 1$ for all
  $i\in\N$. We define $\rho(t)=1/i$ for $t\in (T_{i-1},T_i]$. Then
  $(1/2)\rho(t)\leq \rho(t+h)\leq 2\rho(t)$ for all $h\in[-1,1]$ and
  $t> 1$, and it is easy to check that $\phi_1(t):=\rho(t)\phi_2(t)$
  satisfies the conditions of (ii).
\end{proof}

For $v\in \R^{d+1}$, we write $v=v_1+v_2$, where $v_1\in\R f_1$ and
$v_2\in\Span\{f_2,\dots,f_{d+1}\}$. Let $p(v):=\norm{v_1}$ and $\bar
f_1=\pi(f_1)\in \partial X$. For $T>1$, we define
\begin{align*}
  D_T(\bar f_1,\psi)&=\{x\in X: \norm{x/p(x)-\bar f_1}\leq
  \psi(p(x));\
  \norm{p(x)}\geq T\},\\
  C_T(\bar f_1,\psi)&=\{x\in X: \norm{x/\norm{x}-\bar f_1}\leq
  \psi(\norm{x});\ \norm{x}\geq T\}.
\end{align*}

\begin{lemma}
  \label{prop:CD}
  Let $\psi:(0,\infty)\to (0,\infty)$ be quasi-conformal function such
  that $\psi(t)\to 0$ as $t\to\infty$.  Then there exist $c>1$ and
  $T_0>1$, depending on $\psi$, such that for all $T\geq T_0$,
  \begin{align}
    \label{eq:37}
    D_T(\bar f_1,\psi)&\subset C_{T/2}(\bar f_1,c\psi),\\
    \label{eq:37b}
    C_T(\bar f_1,\psi)&\subset D_{T/2}(\bar f_1, c\psi).
  \end{align}
\end{lemma}

\begin{proof} 
  Let $c>1$ be such that
$$
\psi(ht)\leq (c/2)\psi(t)\quad\hbox{for all $h\in [1/2,2]$.}
$$
For $x=x_1+x_2\in D_T(\bar f_1,\psi)$, we have
$\norm{x/\norm{x_1}-\bar f_1}\leq \psi(\norm{x_1})$. Since
$x_1/\norm{x_1}=\bar f_1$, we have
\[
\norm{x_2}\leq \norm{x_1}\psi(\norm{x_1}),
\]
and
\[
\begin{split}
  \norm{x/\norm{x}-x/\|x_1\|} &\leq
  \bigl\lvert\norm{x_1}-\norm{x}\bigr\rvert/\norm{x_1}\leq
  \norm{x_2}/\norm{x_1}\leq \psi(\norm{x_1}).
\end{split}
\]
Therefore
\begin{equation}
  \label{eq:pi-p}
  \norm{x/\|x\|-\bar f_1}\leq \norm{x/\|x\|-x/\|x_1\|}+\norm{x/\|x_1\|-\bar
    f_1}\leq 2\psi(\norm{x_1}).
\end{equation}
We choose $T_0>0$ such that $\psi(T)<1/2$ for all $T\geq T_0$. Then
$\norm{x_2}/\norm{x_1}<1/2$. Therefore, $\norm{x_1}/\norm{x}\in
(1/2,2)$ and $\psi(\norm{x_1})\leq (c/2)\psi(\norm{x})$. By
\eqref{eq:pi-p}, $x\in C_{T/2}(\bar f_1,\psi)$ and \eqref{eq:37}
follows. Similarly, one proves \eqref{eq:37b}.
\end{proof}

\begin{proof}[Proof of Theorem \ref{thm:vol-est}]
  Applying Lemma \ref{lema:phi} to $\phi(t)=(e^t\psi(e^t))^{d-1}$, we
  may assume that $e^t\psi(e^t)\to 0$ as $t\to\infty$.

  Since the action of $g$ on $X$ preserves the $G$-invariant measure
  on $X$, we have that $\vol(g\,C(v,\psi))=\vol(C(v,\psi))$. Therefore
  due to Lemma~\ref{lema:qc}, in order to prove
  Theorem~\ref{thm:vol-est}, it is sufficient to prove it for a chosen
  $v\in \partial X$. Here we choose $v=\pi(f_1)=\bar f_1$.

  Due to Lemma~\ref{prop:CD}, it is enough to prove that for
  sufficiently large $T$,
  \begin{equation}
    \label{eq:D}
    \vol(D_T(\bar f_1,\psi))=\infty\quad\Longleftrightarrow\quad \int_{0}^\infty
    (e^t\psi(e^t))^{d-1}dt=\infty.
  \end{equation}

  Let $w_0=f_1+(m/2)f_{d+1}\in X$. In what follows we will write the
  matrices of the linear transformation on $\R^{d+1}$ with respect to
  the basis $\{f_1,\dots,f_{d+1}\}$. Let
  \begin{equation}\label{eq:at}
    a_t=\hbox{diag}(e^t,1,\ldots,1,
    e^{-t}).
  \end{equation}
  For $\vs=(\vs_1,\vs_2)\in\R^{d-1}$, where
  $\vs_1=(s_2,\cdots,s_p)\in\R^{p-1}$ and
  $\vs_2=(s_{p+1},\dots,s_{d})\in\R^{d-p}$, let
  \begin{equation}\label{eq:us}
    u(\vs)=\left(
      \begin{matrix}
        1            &   & & \\
        \trn{\vs_1} & I_{p-1}&          & \\
        \trn{\vs_2} &    & I_{d-p}      & \\
        \frac{1}{2}(-\norm{\vs_1}^2+\norm{\vs_2}^2) & -\vs_1 & \vs_2 &
        1
      \end{matrix}
    \right).
  \end{equation}
  Then
  \begin{equation}
    \label{eq:aux}
    a_{t}u(\vs)w_0
    =e^tf_1+(s_2f_2+\dots+s_{d}f_{d})+
    (e^{-t}(m-\norm{\vs_1}^2+\norm{\vs_2}^2)/2)f_{d+1}.
  \end{equation}

  We observe that every $x\in X$ such that $p(x)\ne 0$ can be written
  in the form $x=a_tu(\vs) w_0$. This implies that the set $D_T(\bar
  f_1,\psi)$ consists of $x=a_tu(\vs) w_0$ such that
  \[
  (\norm{f_1} e^{t})\inv\norm{(s_2f_2+\dots+s_{d}f_{d})+
    (e^{-t}(m-\norm{\vs_1}^2+\norm{\vs_2}^2)/2)f_{d+1}} \le
  \psi(e^{t}\norm{f_1})
  \]
  and $\|f_1\|e^t\ge T$.  There exists a constant $c_2>1$ such that
  for any $(s_2,\dots,s_{d+1})\in\R^{d}$, we have
  \[
  \sup(\norm{(s_2,\dots,s_{d})},|s_{d+1}|)/\norm{f_2s_2+\dots+s_df_d+s_{d+1}f_{d+1}}\in(c_2\inv,c_2).
  \]
  Since $\psi$ is quasi-conformal, there exists $c_3>1$ such that
  \[
  \psi(e^t \norm{f_1})/\psi(e^t)\in (c_3\inv,c_3),\quad \forall t>0.
  \]
  Let
  \begin{equation}
    \label{eq:Ut}
    \cU_t(\psi)=
    \left\{\vs\in\R^{d-1}:
      \begin{array}{l}
        \abs{m-\norm{\vs_1}^2+\norm{\vs_2}^2}/2\le e^{2t}\psi(e^t),\\
        \|\vs\|\le e^t\psi(e^t)
      \end{array}
    \right\}.
  \end{equation}
  Then there exists $c>1$ such that
  \begin{align}
    \label{eq:42}
    \bigcup_{t\ge t_0} a_tu(\cU_t(c^{-1}\psi))w_0\subset D_T(\bar
    f_1,\psi)\subset \bigcup_{t\ge t_0} a_tu(\cU_t(c\psi))w_0
  \end{align}
  where $t_0=\log(T/\|f_1\|)$. Since the volume form on $X$ with
  respect to the $(t,\vs)$-coordinates is given by $dtd\vs$, we have
  \[
  \vol_X\left(\bigcup_{t\ge t_0}a_tu(\cU_t(\psi))w_0\right)=\infty
  \iff \int_{t_0}^\infty\int_{\vs\in \cU_t(\psi)}dtd\vs =\infty.
  \]
  Hence to prove \eqref{eq:D}, it suffices to prove that
  \begin{align}
    \label{eq:40}
    \int_0^\infty (e^t\psi(e^t))^{d-1}\, dt=\infty \iff \int_0^\infty
    \vol_{\R^{d-1}}(\cU_t(\psi))\,dt=\infty.
  \end{align}
  Let
  \begin{equation}
    \label{eq:tilda-Ut}
    \tilde \cU_t(\psi)=\{\vs\in\R^{d-1}: \norm{\vs}\leq e^t\psi(e^t)\}.
  \end{equation}
  Let $t_1> 0$ be such that $\psi(e^t)\leq 1$ for all $t\geq t_1$ and
  \begin{align}
    \label{eq:41}
    \mathcal{T}:=\{t>0: e^{2t}\psi(e^t)\leq \abs{m}\}.
  \end{align}
  For $t>t_1$, $t\notin\mathcal{T}$, and $\vs=(\vs_1,\vs_2)\in \tilde
  \cU_t(\psi)$,
  \begin{equation}\label{eq:eq}
    \begin{array}{lcll}
      \abs{m-\norm{\vs_1}^2+\norm{\vs_2}^2} \leq \abs{m}+\norm{\vs}^2\leq \abs{m}+e^{2t}\psi(e^t)^2
      \leq 2e^{2t}\psi(e^t). 
    \end{array}
  \end{equation}
  Hence, for such $t$,
  \[
  \cU_t(\psi)=\tilde \cU_t(\psi)\quad \text{ and }\quad
  \vol_{\mathbb{R}^{d-1}}(
  \cU_t(\psi))=\omega_{d-1}(e^t\psi(e^t))^{d-1}.
  \]
  where $\omega_{d-1}>0$ is the volume of the unit ball in $\R^{d-1}$.
  On the other hand,
  \begin{align}
    \label{eq:Omega1}
    \begin{split}
      \int_{\mathcal{T}} \vol_{\mathbb{R}^{d-1}}(\cU_t(\psi))\,dt&\leq
      \int_{\mathcal{T}}
      \vol_{\mathbb{R}^{d-1}}(\tilde\cU_t(\psi))\,dt\leq \omega_{d-1}\int_{\mathcal{T}} (e^{t}\psi(e^t))^{d-1}\,dt\\
      &\leq \omega_{d-1} \abs{m}^{d-1}\int_{\mathcal{T}}
      e^{-(d-1)t}\,dt<\infty.
    \end{split}
  \end{align}
  This implies \eqref{eq:40} and completes the proof of
  Theorem~\ref{thm:vol-est}.
\end{proof}

\section{Volume estimate for intersections}\label{sec:inter}

Let $G$ be a real algebraic group, and $\sigma$ is an involution of
$G$.  Let $A=\{a_t\}$ be a one-parameter subgroup of $G$ such that
$\sigma(a_t)=a_{-t}$.  We use notation:
\begin{align*}
  &H=\{g\in G:\sigma(h)=h\}, & U^+=\{g\in G:a_{-t} g a_t\to_{t\to\infty} e\},\\
  &Z=Z_G(A), & U^-=\{g\in G:a_tga_{-t}\to_{t\to\infty} e\}.
\end{align*}
Note that $\sigma(U^-)=U^+$, $\sigma(Z)=Z$, and $\sigma(U^+)=U^-$.

We fix a (right) invariant Riemannian metric on $G$. For a subgroup
$S$ of $G$, we set $S_\epsilon=\{s\in S:\, d(s,e)<\epsilon\}$. (It
will be convenient to use that $S_\epsilon^{-1}=S_\epsilon$ and
$S_{\epsilon_1}S_{\epsilon_2}\subset S_{\epsilon_1+\epsilon_2}$.)

The following theorem is the main result of this section:

\begin{theo} \label{thm:vol-B_t} There exist constants $r_0,t_0>0$
  such that for any measurable subsets $\Psi_i\subset U^-_{r_0}$ and
  $g\in G$, setting $D_i:=\Psi_iZ_{r_0}H_{r_0}g\Gamma$, $i=1,2$, we
  have
  \[
  \vol_{G/\Gamma}(D_1\cap a_t D_2)\leq
  C\vol_{U^{-}}(\Psi_1)\vol_{U^-}(\Psi_2), \quad \forall t\geq t_0.
  \]
  for some $C=C(g)>0$.
\end{theo}

The proof of Theorem \ref{thm:vol-B_t} consists of two main steps: in
Proposition \ref{thm:inter-vol}, we estimate the volumes of the
intersections of lifts of $D_1$ and $a_tD_2$ in $G$, and using
Proposition \ref{p:lifts} with Lemma \ref{lema:UZaV}, we estimate the
number of lifts which intersect.

We start the proof by introducing convenient coordinate systems in
$G$.  Let $\la{g}$, $\la{a}$, $\la{h}$, $\la{z}$, $\la{u}^+$,
$\la{u}^-$ denote the corresponding Lie algebras.  We have the
decomposition
$$
\la{g}=\la{h}\oplus\la{q}
$$
into $(+1)$- and $(-1)$-eigenspaces of $\sigma$.  Since $\Ad(a_t)$ is
skew-symmetric with respect to the form $\langle X,\sigma(X)\rangle$,
$X\in\la{g}$, it follows that $\Ad(a_t)$ is diagonalizable and we have
the decomposition:
\begin{align*}
  \la{g}=\la{u}^-\oplus\la{z}\oplus \la{u}^+.
\end{align*}
Hence, the product map $U^-\times Z\times U^+\to G$ is a
diffeomorphism in a neighborhood of $e$, and there exist $r_0>0$ and
analytic maps $\vu^-:G_{r_0}\to U^-$, $\vz:G_{r_0}\to Z$,
$\vu^+:G_{r_0}\to U^+$ such that every element $g\in G_{r_0}$ can be
uniquely written as
\begin{equation}\label{eq:uzu}
  g=\vu^-(g)\vz(g)\vu^+(g).
\end{equation}

For every $x\in \la{u}^+$, we have $x=-\sigma(x)+(x+\sigma(x))$ where
$\sigma(x)\in\la{u}^-$ and $x+\sigma(x)\in\la{h}$. Hence, we also have
the decomposition:
\begin{equation}\label{eq:dec0}
  \la{g}=\la{u}^-\oplus (\la{z}+\la{h})
\end{equation}
and the decomposition:
\begin{align*} 
  \la{g}=\la{u}^-\oplus\la{z}_{\la{q}}(\la{a}) \oplus \la{h}.
\end{align*}
Let $B=\exp(\la{z}_\la{q}(a))$.  It follows that there exist $r_0>0$
and analytic maps $\vv:G_{r_0}\to U^-$, $\vb:G_{r_0}\to B$,
$\vh:G_{r_0}\to U^+$ such that every element $g\in G_{r_0}$ can be
uniquely written as
\begin{equation}\label{eq:vbh}
  g=\vv(g)\vb(g)\vh(g).
\end{equation}

\begin{prop} \label{thm:inter-vol} There exist $r_0,t_0,c>0$ such that
  for measurable subsets $\Psi_1,\Psi_2\subset U^-_{r_0}$, $t>t_0$,
  and $g\in G$, we have
  \begin{equation*}
    \vol_G(\Psi_1Z_{r_0}H_{r_0} g\cap\, a_t\Psi_2Z_{r_0}H_{r_0}) \leq
    c\vol_{U^-}(\Psi_1)\vol_{U^-}(\alpha_t(\Psi_2))
  \end{equation*}
  where $\alpha_t(g)=a_tga_{-t}$.
\end{prop}

\begin{proof}
  Let $Y:=\Psi_1Z_{r_0}H_{r_0} g \cap\, a_t\Psi_2Z_{r_0}H_{r_0}$ and
  $y\in Y$.  Let $\zeta_i\in\Psi_iZ_{r_0}H_{r_0}$, $i=1,2$, be such
  that $y=\zeta_1 g=a_t\zeta_2$.  Then
  \begin{equation*}
    Y=(\Psi_1Z_{r_0}H_{r_0}\zeta_1\inv\cap \, \alpha_t(\Psi_2Z_{r_0}H_{r_0}\zeta_2\inv))y.
  \end{equation*}
  We express $\zeta_i=u_iz_ih_i$, where $u_i\in\Psi_i$, $z_i\in
  Z_{r_0}$ and $h_i\in H_{r_0}$.  Then $H_{r_0}h_i\inv\subset
  H_{2r_0}$ and
  \begin{equation*}
    \begin{split}
      Y&\subset (\Psi_1Z_{r_0}H_{2r_0}(u_1z_1)\inv\cap \,
      \alpha_t(\Psi_2Z_{r_0}H_{2r_0}
      (u_2z_2)\inv))y\\
      &\subset (\Psi_1Z_{r_0}H_{2r_0}\zeta
      \cap\,\alpha_t(\Psi_2Z_{r_0}H_{2r_0}))y_1.
    \end{split}
  \end{equation*}
  where $\zeta=(u_1z_1)^{-1}\alpha_t(u_2z_2)$ and
  $y_1=\alpha_t(u_2z_2)^{-1} y$.  Since the measure is right
  translation invariant, it remains to show that for the set
  \begin{equation}
    \label{eq:X}
    X:=\Psi_1Z_{r_0} H_{2r_0} \zeta\cap \alpha_t(\Psi_2 Z_{r_0} H_{2r_0}),
  \end{equation}
  we have
  \begin{equation}
    \label{eq:volX}
    \vol_G(X)\leq c\,\vol_{U^-}(\Psi_1)\vol_{U^-}(\alpha_t(\Psi_2)).
  \end{equation}
  Note that $\alpha_t|_{U^-Z}$ is Lipschitz (uniformly on $t$). Hence,
  there exists $l>0$ such that $\zeta\in G_{l r_0}$.

  For small $r>0$, we have a well-defined map
$$
\vp=\vu^-\times \vz\times \vu^+:G_r\to U^-\times Z\times U^+
$$
which is a diffeomorphism onto its image. For $\zeta,g\in G$ close to
identity, we also have the map $\vp_\zeta(g)=\vp(g\zeta\inv)$.  Note
that
\begin{equation}\label{eq:conv1}
  \vp_\zeta(g)\to \vp(g)\quad\hbox{and}\quad D(\vp_\zeta)_g\to  D(\vp)_g\quad\hbox{as $\zeta\to e$,}
\end{equation}
uniformly for $g$ in a neighborhood of identity.  Given $g=u^-zu^+\in
U^-ZU^+$ with components close to identity, we write
\begin{align*}
  \begin{split}
    g=u^-zu^+&=u^-z\vv(u^+)\vb(u^+) \vh(u^+)\\
    &=\psi(u^-,z,u^+) \cdot \eta(u^-,z,u^+) \cdot \vh(u^+)
  \end{split}
\end{align*}
where
\begin{align*}
  \psi(u^-,z,u^+)&:=u^- z\vv(u^+)z^{-1} \in U^-,\\
  \eta(u^-,z,u^+)&:=z\vb(u^+)\in Z.
\end{align*}
We claim that if $u_1z_1h_1=u_2z_2h_2\in U^-ZH$ with components close
to identity, then $u_1=u_2$. Indeed, we have
$$
z_1^{-1}(u_2^{-1}u_1)z_1= (z_1^{-1}z_2) (h_2h_1^{-1})\in U^-\cap ZH.
$$
Since the map $u\mapsto z_1^{-1}uz_1$, $u\in U^-$, is Lipschitz in a
neighborhood of identity, and $U^-$ is transversal to $ZH$ at identity
(see (\ref{eq:dec0})), it follows that $u_1=u_2$.  In particular, we
note that
\begin{equation}
  \label{eq:X1}
  g\in \Psi_1Z_{r_0}H_{2r_0}\zeta\implies \psi(\vp_\zeta(g))\in\Psi_1.  
\end{equation}

For $g=u^-zu^+\in U^-ZU^+$ with components close to identity, we write
\begin{align*}
  \begin{split}
    g=u^-zu^+
    &=u^-z\alpha_t(\alpha_{-t}(u^+))\\
    &=u^-z\alpha_t(\vv(\alpha_{-t}(u^+))\vb(\alpha_{-t}(u^+))\vh(\alpha_{-t}(u^+)))\\
    &= \phi_t(u^-,z,u^+)\cdot z\vb(\alpha_{-t}(u^+))
    \cdot \alpha_t(\vh(\alpha_{-t}(v))),\\
    &\in U^-\cdot Z\cdot \alpha_t(H),
  \end{split}
\end{align*}
where
\begin{equation*}
  \phi_t(u^-,z,u^+):=u^- z\alpha_t(\vv(\alpha_{-t}( u^+))) z^{-1}\in U^-.
\end{equation*}
By the argument as above, we have
\begin{equation}
  \label{eq:X2}
  g\in \alpha_t(\Psi_2)Z_{r_0}\alpha_t(H_{2r_0})\implies \phi_t(\vp(g))\in\alpha_t(\Psi_2). 
\end{equation}
 
As $t\to\infty$, the map $u^+\mapsto\alpha_t(\vv(\alpha_{-t}( u^+)))$
converges in $C^1$-topology to the constant function $e$. Therefore,
setting $\phi(u^-,z,u^+)=u^-$, we have
\begin{equation}\label{eq:conv2}
  \phi_t(\vp(g))\to\phi(\vp(g))\quad\hbox{and}\quad D(\phi_t)_{\vp(g)}\to D(\phi)_{\vp(g)}\quad \hbox{as $t\to\infty$},
\end{equation}
uniformly on $g$ in a neighborhood of identity.

For small $r>0$, consider the maps $\Phi_{\zeta,t}, \Phi: G_{r}\to
U^-\times Z \times U^-$ defined by
\begin{equation*}
  \begin{split}
    \Phi_{\zeta,t}(g)
    &=(\psi(\vp_\zeta(g)),\eta(\vp_\zeta(g)),\phi_t(\vp(g))),\\
    \Phi(g)&=(\psi(\vp(g)),\eta(\vp(g)),\phi(\vp(g))).
  \end{split}
\end{equation*}
By (\ref{eq:conv1}) and (\ref{eq:conv2}),
\begin{equation}\label{eq:conv3}
  \Phi_{\zeta,t}(g)\to \Phi(g)\quad\hbox{and}\quad D(\Phi_{\zeta,t})_g \to D(\Phi)_g\quad\hbox{
    as $(\zeta,t)\to (e,\infty)$,}
\end{equation}
uniformly on $g$ in a neighborhood of identity. We have
\begin{equation*} 
  D(\Phi\circ \vp^{-1})_{(e,e,e)}
  = \left(
    \begin{tabular}{ccc}
      id & 0 & $\left(\frac{\partial \vv}{\partial u^+}\right)_e$ \\
      0 & id & $\left(\frac{\partial \vb}{\partial u^+}\right)_e$ \\
      id & 0 & 0
    \end{tabular}
  \right).
\end{equation*}
Since every $x\in\la{u}^+$ can be written as
$$
x=-\sigma(x)+0+(x+\sigma(x))\in \la{u}^-\oplus \la{b}\oplus \la{h},
$$
it follows that $\left(\frac{\partial \vv}{\partial
    u^+}\right)_e=-\sigma$ and $\abs{\det D(\Phi\circ
  \vp^{-1})_{(e,e,e)}}=1$.  Hence, by (\ref{eq:conv3}), there exists
$c_1>0$ such that for every $\zeta\in G$ close to $e$, sufficiently
large $t$, and $g\in G$ in a neighborhood of $e$, we have
\begin{equation}\label{eq:det}
  \abs{\det D(\Phi_{\zeta,t})_g} > 
  c_1.
\end{equation}
By \eqref{eq:X}, \eqref{eq:X1} and \eqref{eq:X2},
\begin{equation*}
  \Phi_{\zeta,t}(X)\subset \Psi_1\times Z_{\kappa r_0}\times \alpha_t(\Psi_2)
\end{equation*}
for some $\kappa>0$. Now taking $r_0>0$ sufficiently small,
\eqref{eq:volX} follows from (\ref{eq:det}). This proves the
proposition.
\end{proof}

\begin{prop}\label{p:lifts}
  There exist $r_0>0$ and $d>0$ such that for every $t>0$,
$$
\#(\Gamma\cap G_{r_0} a_t G_{r_0})\le d \det
(\Ad(a_t)|_{\mathfrak{u}^+}).
$$
\end{prop}

To prove this proposition, we use the following lemma:

\begin{lemma}
  \label{lema:UZaV}
  There exist $\epsilon_0>0$ and $l>1$ such that for every $r\in
  (0,\epsilon_0)$ and $t>0$,
  \begin{equation*}
    \label{eq:UZaV}
    G_{r} a_t G_{r} \subset U^-_{l r}Z_{l r}a_t U^+_{l r}.
  \end{equation*}
\end{lemma}

\begin{proof}  
  There exists $l_1>1$ such that for every small $r>0$, $G_r\subset
  U^-_{l_1r}Z_{l_1r}U^+_{l_1r}$ (see (\ref{eq:uzu})).  Since the map
  $u\mapsto a_tua_t^{-1}$, $u\in U^-$, is Lipschitz on compact sets,
  there exists $\kappa>0$ such that $a_t U^-_r a_t^{-1}\subset
  U^-_{\kappa r}$.  Therefore,
  \begin{equation*}
    G_ra_tG_r \subset G_ra_tU^-_{l_1r}Z_{l_1r}U^+_{l_1r}\subset G_rU^-_{\kappa l_1r}Z_{l_1r}a_tU^+_{l_1r}
    \subset G_{l_2 r}a_tU^+_{l_2r}.
  \end{equation*}
  where $l_2=1+(\kappa+1)l_1$. Similarly,
$$
G_{r}a_tU^+_{r}\subset U^-_{l_1r}Z_{l_1r}U^+_{l_1r}a_tU^+_r \subset
U^-_{l_1r}Z_{l_1r}a_t U^+_{\kappa l_1r}U^+_r \subset
U^-_{l_1r}Z_{l_1r}a_t U^+_{(\kappa l_1+1)r}.
$$
This implies the claim.
\end{proof}

\begin{proof}[Proof of Proposition \ref{p:lifts}]
  Let $\epsilon>0$ be such that
$$
G_\epsilon\gamma_1\cap G_\epsilon\gamma_2 =\emptyset\quad\hbox{for
  $\gamma_1\ne \gamma_2$}.
$$
Then
$$
\#(\Gamma\cap G_{r} a_t G_{r})\le \vol_G(G_\epsilon G_{r} a_t
G_{r})/\vol_G(G_\epsilon).
$$
Hence, it follows from Lemma \ref{lema:UZaV}, that for some
$c_1,r_1>0$
$$
\#(\Gamma\cap G_{r} a_t G_{r})\le c_1\,\vol_G(U^-_{r_1}Z_{r_1}a_t
U^+_{r_1}).
$$
Since the Haar measure in $U^-ZU^+$-coordinates is given by
$$
\det(\Ad(z)|_{\mathfrak{u}^+})du^-dzdu^+,
$$
the claim follows.
\end{proof}

\begin{proof}[Proof of Theorem~\ref{thm:vol-B_t}]
  Let $r_0>0$ be sufficiently small.  We have
  \begin{equation}\label{eq:sum}
    \vol_{G/\Gamma} (D_1\cap a_t D_2)
    \leq \sum_{\gamma\in\Gamma} \vol_G(\Psi_1Z_{r_0}H_{r_0} g\gamma\cap
    a_t\Psi_2Z_{r_0}H_{r_0} g).
  \end{equation}
  If $\gamma\in\Gamma$ satisfies
$$
\Psi_1Z_{r_0}H_{r_0} g\gamma\cap a_t\Psi_2Z_{r_0}H_{r_0}
g\neq\emptyset,
$$
then
$$
g \gamma g\inv \in (\Psi_1Z_{r_0}H_{r_0})^{-1} a_t
\Psi_2Z_{r_0}H_{r_0} \subset G_{3r_0}a_t G_{3r_0}.
$$
Hence, by Proposition \ref{p:lifts}, the number of terms in the sum
(\ref{eq:sum}) is bounded by $d \det (\Ad(a_t)|_{\mathfrak{u}^+})$ for
some $d=d(g)>0$.  Applying Proposition \ref{thm:inter-vol}, we deduce
from (\ref{eq:sum}) that
$$
\vol_{G/\Gamma} (D_1\cap a_t D_2) \leq d \det
(\Ad(a_t)|_{\mathfrak{u}^+})\cdot
c\vol_{U^-}(\Psi_1)\vol_{U^-}(\alpha_t(\Psi_2)).
$$
This proves the theorem.
\end{proof}

In the rest of this section we compute the asymptotics for the number
of lattice points in the boxes $U^-_{r_1}Z_{r_2}a_t U^+_{r_3}$ as
$t\to\infty$. This result is of independent interest, but it is not
needed in the proof of the main theorem.  We will assume that the
action of $\{a_t\}$ on $G/\Gamma$ is mixing. Due to the Howe--Moore
theorem~\cite{Howe+Moore:vanishing} on vanishing of matrix
coefficients, the mixing condition is satisfied if $\Gamma$ is an
irreducible lattice in $G$. For example, this irreducibility condition
is satisfied if $G$ is a connected noncompact simple Lie group.

\begin{theo}
  \label{thm:gamma}
  For every $r_1,r_2,r_3>0$,
  \begin{equation*}
    \begin{split}
      \card{\Gamma \cap U^-_{r_1}Z_{r_2}a_tU^+_{r_3}}
      \sim_{t\to\infty}
      \frac{\vol_G(U^-_{r_1}Z_{r_2}a_tU^+_{r_3})}{\vol_{G/\Gamma}(G/\Gamma)}
      = \frac{\lambda(\vr)\det(\Ad
        a_t|_{\la{U}^+})}{\vol_{G/\Gamma}(G/\Gamma)},
    \end{split}
  \end{equation*}
  where
  \[
  \lambda(\vr)=\vol_{U^-}(U^-_{r_1})\vol_{U^+}(U^+_{r_3})\int_{Z_{r_2}}\rho(z)\,dz,
  \]
  $\rho(z)=\abs{\det(\Ad z|_{\la{U^+}})}$, and $dz$ denotes the Haar
  integral on $Z$ associated to $\vol_Z$.
\end{theo}

We will prove this theorem using mixing of the flow $a_t$ (as in
\cite{DRS,Esk+McM:mixing}) and the following lemma:

\begin{lemma}
  \label{lema:perturb}
  For every $r_0>0$, there exist $l,\epsilon_0>0$ such that if $u\in
  U^-_{r_0}$, $z\in Z_{r_0}$ and $v\in U^+_{r_0}$, then for any $g\in
  G_s$ with $s\in(0,\epsilon_0)$ and $t>0$, we have
  \[
  g(uza_tv)=(uu_1)(z_1z)a_t (v_1v),
  \]
  where $u_1\in U^-_{ls}$, $z\in Z_{ls}$, and $v_1\in U^+_{ls}$.
\end{lemma}

\begin{proof}
  We have
$$
g(uza_tv)=u(g_1a_t)zv
$$
where $g_1=u^{-1}gu\in G_{l_1 s}$ for some $l_1=l_1(r_0)>0$. By Lemma
\ref{lema:UZaV},
$$
g_1a_t=u_1z_1a_tv_1
$$
where $u_1\in U_{l_2s}^-$, $z_1\in Z_{l_2 s}$ and $v_1\in U^+_{l_2s}$
for some $l_2=l_2(r_0)>0$.  Hence,
$$
g(uza_tv)=(uu_1)(z_1z)a_t(v_2v)
$$
where $v_2=z\inv v_1 z\in U^+_{l_3s}$ for some $l_3=l_3(r_0)>0$.  This
implies the claim.
\end{proof}

\begin{proof}[Proof of Theorem~\ref{thm:gamma}]
  Let $du$, $dz$ and $dv$ denote the Haar integrals on $U^-$, $Z$ and
  $U^+$, respectively. A Haar measure on $G$ is defined by
  \begin{equation}
    \label{eq:vol-form}
    \int_G f\,d\mu = \int_{U^-} \int_Z \int_{U^+}
    f(uzv)\rho(z)\,dudzdv,\quad
    \forall f\in \Cc(G).
  \end{equation}
  Now given $\vr=(r_1,r_2,r_3)$, we put
  $E_t(\vr)=U^-_{r_1}Z_{r_2}a_tU^+_{r_3}$.  Then
  \begin{equation} \label{eq:vol-form2}
    \mu(E_t(\vr))=\vol_{U^-}(U^-_{r_1})\left(\int_{Z_{r_2}}
      \rho(z)\,dz\right)\rho(a_t)\vol_{U^+}(U^+_{r_3})\\
    =\lambda(\vr)\rho(a_t).
  \end{equation}

  Let $l,\epsilon_0>0$ be as in Lemma~\ref{lema:perturb}. We use
  parameters $s\in (0,\epsilon_0)$, $r_i^\pm=r_i\pm ls$, and
  $\vr^\pm=(r_1^\pm,r_2^\pm,r_3^\pm)$.  Then by
  Lemma~\ref{lema:perturb}, for every $t>0$,
  \begin{equation}\label{eq:rm}
    E_t(\vr^-)\subset \bigcap_{g\in G_s} gE_t(\vr)\subset
    E_t\subset \bigcup_{g\in G_s} gE_t(\vr)\subset E_t(\vr^+).
  \end{equation}
  
  Let $\bar\mu$ denote the finite $G$-invariant measure on $G/\Gamma$
  associated to $\mu$. By our assumption the action of $\{a_t\}_{t>0}$
  by left translations on $G/\Gamma$ is mixing. In other words, if we
  put $y_0=e\Gamma$, then given any $\phi\in\Cc(G/\Gamma)$ and small
  $r>0$,
  \begin{equation*}
    \frac{1}{\mu(U^-_{r}Z_{r}U^+_{r_3})}\int_{U^-_{r}}\int_{Z_r}
    \int_{U^+_{r_3}}\phi(a_tuzvy_0)\rho(z)\,dudzdv\to \frac{1}{\bar\mu(G/\Gamma)}
    \int_{G/\Gamma} \phi\,d\bar\mu
  \end{equation*}
  at $t\to\infty$.  Since $a_tuza_t\inv\to z$ as $t\to\infty$, and
  since $\rho(z)\to e$ as $z\to e$, from the uniform continuity of
  $\phi$, we deduce that
  \begin{equation*}
    \label{eq:U+}
    \lim_{t\to\infty} 
    \frac{1}{\vol_{U^+}(U^+_{r_3})}\int_{U^+_{r_3}}\phi(a_tvy_0)\,dv
    =\frac{1}{\bar\mu(G/\Gamma)} \int_{G/\Gamma} \phi\,d\bar\mu.
  \end{equation*}
  Hence, in view of \eqref{eq:vol-form} and (\ref{eq:vol-form2}),
  \begin{equation}
    \label{eq:33}
    \lim_{t\to\infty}
    \frac{1}{\rho(a_t)}\int_{E_t(\vr)}\phi(gy_0)\,d\mu(g)
    =\frac{\lambda(\vr)}{\bar\mu(G/\Gamma)} 
    \int_{G/\Gamma} \phi\,d\bar\mu.
  \end{equation}
  Now as in \cite{DRS,Esk+McM:mixing}, we introduce functions on
  $G/\Gamma$:
  \begin{align*}
    \label{eq:29}
    F_t(gy_0)=\sum_{\gamma\in\Gamma} \chi_{E_t(\vr)}(g\gamma)
    \quad\text{and} \quad F_t^\pm(gy_0)=\sum_{\gamma\in\Gamma}
    \chi_{E_t(\vr^\pm)}(g\gamma).
  \end{align*}
  We note that $F_t(y_0)=\card{\Gamma \cap E_t(\vr)}$. Let
  $\phi\in\Cc(G/\Gamma)$ with $\supp(\phi)\subset G_sy_0$ and
  $\int_{G/\Gamma}\phi\,d\bar\mu=1$. Then by (\ref{eq:rm}),
  \begin{equation}
    \label{eq:31}
    \int_{G/\Gamma} F_t^-(y)\phi(y)\,d\bar\mu(y) \leq F_t(y_0)\leq
    \int_{G/\Gamma} F_t^+(y)\phi(y)\,d\bar\mu(y).  
  \end{equation}
  Using (\ref{eq:33}), we obtain
  \begin{equation}
    \label{eq:30}
    \begin{split}
      &\frac{1}{\rho(a_t)} \int_{G/\Gamma}
      F_{t}^\pm(y)\phi(y)\,d\bar\mu(y) =\int_{G/\Gamma}
      \sum_{\gamma\in\Gamma}\chi_{E_t^\pm(\vr)}(g\gamma)\phi(gy_0)\,d\bar\mu(g\Gamma)\\
      =&\frac{1}{\rho(a_t)}\int_{G}\chi_{E_t(\vr^\pm)}(g)
      \phi(gy_0)d\mu(g) \to
      \frac{\lambda(\vr^\pm)}{\bar\mu(G/\Gamma)}\quad\hbox{as
        $t\to\infty$.}
    \end{split}
  \end{equation}
  Since $\lambda(\vr^+)/\lambda(\vr^-)\to 1$ as $s\to 0$, from
  \eqref{eq:31} and \eqref{eq:30}, we conclude that
  \[
  \card{\Gamma \cap E_t(r)}=F_t(y_0)\sim_{t\to\infty}
  \frac{\lambda(\vr)\rho(a_t)}{\bar\mu(G/\Gamma)}
  \]
  as required.
\end{proof}

\section{Proof of the main theorems}\label{sec:proof}

To prove Theorem~\ref{thm:bcl}, we use a converse of the
Borel-Cantelli lemma. It is well-known that such a converse holds
under some quasi-independence condition. We will use the following
version (see \cite[\S1]{Sul}, and also \cite[Lemma~2.3]{har},
\cite[Lemma~5]{Spr} for more general results):

\begin{prop}
  \label{prop:dynbcl}
  Let $(Y,\mu)$ be a finite measure space. Let $\{F_n\}_{n\in\N}$ be a
  sequence of measurable subsets of $Y$ such that $\sum_{n=1}^\infty
  \mu(F_n)=\infty$. Suppose there exist $n_0\in\N$ and a constant
  $C>0$ such that
  \begin{equation}
    \label{eq:25}
    \mu(F_n\cap F_m)\leq C\mu(F_n)\mu(F_m), \quad \forall
    m,n\in\N,\ \abs{m-n}\geq n_0.
  \end{equation}
  Let $F=\cap_{n\in\N}(\cup_{m\geq n} F_m)$. Then $\mu(F)>0$.
\end{prop}

\begin{proof}[Proof of Theorem~\ref{thm:bcl}]
  First we suppose that
  \[
  \sum_{n=1}^\infty \vol_{U^-}(\Psi_n)=\infty \quad\text{and}\quad
  t_{n+1}-t_n\geq \delta_0>0, \ \forall n\in\N.
  \]
  Let $r_0>0$ be as in Theorem~\ref{thm:vol-B_t}.  Since $B$ has
  positive measure, there exist $z\in Z$ and $h\in H$ such that for
  every $r_0>0$, the set $B\cap zZ_{r_0}H_{r_0}h$ has positive measure
  as well.  Let $B_0:=z^{-1}Bh^{-1}\cap Z_{r_0}H_{r_0}$. We consider
  the sets
$$
F_n=a_{t_n}\Psi_{n}zB_0 hy_0=z a_{t_n}(z^{-1}\Psi_{n}z)B_0 hy_0\subset
G/\Gamma.
$$
Recall that we are assuming that the sets $\Psi_n$ are contained in a
small neighborhood of identity. Hence, taking $r_0$ sufficiently
small, the sets $(z^{-1}\Psi_{n}z)B_0 h$ project injectively on
$G/\Gamma$, and
$$
\sum_{n=1}^\infty \vol_{G/\Gamma}(F_n)=\infty.
$$
By Theorem~\ref{thm:vol-B_t}, there exist $n_0\in \N$ and $C_1,C_2>0$
such that
\[
\vol_{G/\Gamma}(F_k\cap F_l)\leq C_1
\vol_{U^-}(\Psi_k)\vol_{U^-}(\Psi_l)\leq C_2
\vol_{G/\Gamma}\vol(F_k)\vol_{G/\Gamma}(F_l)
\]
for all $k,l\in\N$ such that $\abs{k-l}\geq n_0$.  Let
$F=\cap_{n\in\N}\cup_{m\geq n} F_m$.  Then by
Proposition~\ref{prop:dynbcl} applied to $Y=G/\Gamma$, we conclude
that $\vol_{G/\Gamma}(F)>0$. Now for any $y\in F$, we have that
\begin{equation}\label{eq:card}
  \card{\{n\in\N: y\in a_{t_n}\Psi_nBy_0\}}\geq \card{\{n\in \N: y\in
    F_n\}}=\infty.
\end{equation}
This proves the first part of the theorem.

To prove the second part, we assume that $\sum_{n=1}^\infty
\vol_G(\Psi_n)< \infty$. If we set $F_n=a_{t_n}\Psi_nBy_0$, then
\[
\sum_{n=1}^\infty \vol_{G/\Gamma}(F_n)<\infty.
\]
Therefore by the Borel-Cantelli Lemma for almost all $y\in G/\Gamma$,
we have that $y\in F_n$ for only finitely many $n\in\N$. This proves
the second part of the theorem.
\end{proof}

\begin{proof}[Proof of Corollary~\ref{cor:bcl}]
  If we put $t_n=nt_0$, the condition of the first part of
  Theorem~\ref{thm:bcl} is satisfied; and let the notation be as in
  the proof of this part given as above.  Since $\Psi_{n+1}\subset
  \Psi_{n}$ for all $n\in\N$, we have $F_{n+1}\subset T(F_n)$, and
  $F\subset T(F)$. Hence, since $T$ is ergodic, the set $F$ has full
  measure.  Now the claim follows from \eqref{eq:card}.
\end{proof}

To prove Theorem \ref{thm:main-qform}, we will need the following:

\begin{prop} \label{prop:Gamma} The $\Gamma$-action on $\partial X$ is
  ergodic with respect to $\mu_\infty$.
\end{prop}

\begin{proof}
  We denote by $G^0$ the connected component of identity of $G$.  The
  space $\partial X$ is not connected in general, but it consists of
  at most two connected components, which are mapped to each other by
  the transformation $x\mapsto -x$. Since this transformation is in
  $\Gamma$, it suffices to show that $G^0\cap \Gamma$ acts ergodically
  on the connected components of $\partial X$. Each connected
  component can be identified with a homogeneous space $G^0/P$ of
  $G^0$ where $P$ is a closed noncompact algebraic subgroup of
  $G^0$. Note that $G^0$ is a connected simple Lie group unless the
  signature of the quadratic form is $(2,2)$, and in the later case,
  $G^0$ is semisimple and one can check that the projection of $P$ to
  the nontrivial simple factors of $G^0$ are noncompact.  Hence, by
  Mautner's lemma, the $P$-action on $G^0/(G^0\cap \Gamma)$ is ergodic
  with respect to the $G^0$-invariant probability measure.  Therefore
  $(G^0\cap \Gamma)$-action on $G^0/P$ is ergodic with respect to the
  $G^0$-semi-invariant probability measure (see, for example,
  \cite{Furst:uniq}).
\end{proof}

Now we begin the proof of Theorem \ref{thm:main-qform}.  We use
notation as in \S\ref{sec:vol}. In particular, we recall that
$G=\hbox{O}(Q)$, $A=\{a_t\}$ is the one-parameter subgroup defined in
(\ref{eq:at}), and $w_0=(f_1+(m/2)f_2)\in X$.  Let
$H=\hbox{Stab}_G(w_0)$. Note that $H$ is the set of fixed points of
the involution $\sigma(g)= s_0gs_0$, where $s_0\in \hbox{O}(Q)$ is
given by
$$
s_0:\;\; f_1\mapsto -\frac{m}{2}f_{d+1},\;\; f_{d+1}\mapsto
-\frac{2}{m}f_1,\;\; f_i\mapsto f_i,\; i=2,\ldots,d.
$$
Moreover, $\sigma(a_t)=a_{-t}$.  Let $\Gamma=G(\mathbb{Z})$, which is
a lattice in $G$ by the Borel--Harish-Chandra theorem.  Hence, we are
in the setting of Theorem~\ref{thm:bcl}.  Note that $U^-=\{u(\vs):\,
\vs\in\R^{d-1}\}$, where $u(\vs)$ is defined in (\ref{eq:us}).

Let $x_0\in X(\Z)$ and $g_0\in G$ be such that $w_0=g_0x_0$.

\begin{proof}[Proof of Theorem \ref{thm:main-qform}(i)]
  Suppose that
$$
\int_1^\infty t^{d-2}\psi(t)^{d-1}\, dt=\int_0^\infty
(e^t\psi(e^t))^{d-1}\,dt=\infty.
$$ 
Then by Lemma~\ref{lema:phi}(ii), there exists a measurable
quasi-conformal function $\psi_1\leq \psi$ such that
$$
e^t\psi_1(e^t)\to_{t\to\infty} 0,\quad
\frac{\psi(t)}{\psi_1(t)}\to_{t\to\infty} \infty,\quad \int_0^\infty
(e^t\psi_1(e^t))^{d-1}\,dt=\infty.
$$
In view of Lemma~\ref{prop:CD} and \eqref{eq:42}, there exist
$c,t_2>0$ such that
\begin{equation}
  \label{eq:12}
  \bigcup_{t>t_2} a_tu(\cU_t(\psi_1))w_0\subset C(\bar f_1,c\psi_1).
\end{equation}
(Recall that $\cU_t(\psi_1)\subset\R^{d-1}$ is as defined in
\eqref{eq:Ut}.)  Since $\psi_1$ is quasi-conformal, we have
$$
\sum_{k= 1}^\infty (e^k\psi_1(e^k))^{d-1}=\infty,
$$
and by the discrete version of (\ref{eq:40}),
$$
\sum_{k=1}^\infty \vol_{\mathbb{R}^{d-1}}(\cU_{k}(\psi_1))=\infty.
$$
For $p>0$, let $\mathcal{T}_p=\mathcal{T}+(-p,p)$, where $\mathcal{T}$
is defined in (\ref{eq:41}).  By the same argument as in
(\ref{eq:Omega1}),
$$
\sum_{k\in\mathcal{T}_p\cap\mathbb{N}}
\vol_{\mathbb{R}^{d-1}}(\cU_{k}(\psi_1))<\infty.
$$
It follows that there exists a sequence $k_n\in\mathbb{N}\backslash
\mathcal{T}_p$, $k_n\to\infty$, such that
\begin{equation*}
  \sum_{n=1}^\infty \vol_{\mathbb{R}^{d-1}}(\cU_{k_n}(\psi_1))=\infty.
\end{equation*}

Let $r_0>0$ be as in Theorem~\ref{thm:vol-B_t}, and $p>0$ is such that
$A_p:=\{a_h:\abs{h}<p\}\subset G_{r_0}$. Define
\[
\cV_{t}(\psi_1)=\bigcap_{\abs{h}<p} e^{-h}\cU_{t+h}(\psi_1).
\]
Since $k_n\notin \mathcal{T}_p$, it follows from (\ref{eq:eq}) that
for sufficiently large $n$, $\cU_{k_n+h}(\psi_1)=\tilde
\cU_{k_n+h}(\psi_1)$ when $|h|<p$, where $\cU_t(\psi_1)$ is defined in
(\ref{eq:tilda-Ut}).  Hence, using that $\psi_1$ is quasi-conformal,
we deduce that
\[
\cV_{k_n}(\psi_1)=\bigcap_{\abs{h}<p}e^{-h}\tilde
\cU_{k_n+h}(\psi_1)=\bigcap_{\abs{h}<p} \tilde
\cU_{k_n+h}(e^{-h}\psi_1)\supset \tilde \cU_{k_n}(c_1\psi_1)
\]
for some $c_1>0$.  In particular,
\begin{equation}
  \label{eq:barUn}
  \sum_{n=1}^\infty \vol_{\mathbb{R}^{d-1}} {(\cV_{k_n}(\psi_1))}=\infty.
\end{equation}

Let $\Psi_n=u(\cV_{k_n}(\psi_1))$, $B=A_pH_{r_0}$, and $B_n=\Psi_nB$
for all $n\in\N$.  One can check that $B$ is open in $ZH$, and in
particular, it has positive measure.  By Theorem~\ref{thm:bcl}, the
set
\[
E=\{z\in G/\Gamma: \#\{n\in\N:z\in a_{k_n}B_ng_0\Gamma\}=\infty\}.
\]
has positive measure.  It follows that the set $\tilde E=\{g\in
G:g\Gamma\in E\}$ has positive measure as well.

Let $g\in \tilde E$. There exist infinitely many $n\in\N$ such that
$g\Gamma\cap a_{k_n}B_ng_0\Gamma\neq\emptyset$.  Hence, there are
infinitely many elements of $\Gamma$ in the set $\cup_{n\ge 1} g^{-1}
a_{k_n}B_ng_0$. By \eqref{eq:12},
\begin{equation*}
  \begin{split}
    a_{k_n}B_ng_0x_0&=a_{k_n}B_nw_0=a_{k_n}\Psi_nA_pw_0 =\bigcup_{\abs{h}<p} a_{k_n+h}u(e^h\cV_{k_n}(\psi_1))w_0\\
    &\subset \bigcup_{\abs{h}<p} a_{k_n+h}u(\cU_{k_n+h}(\psi_1))w_0
    \subset C(\bar f_1,c\psi_1).
  \end{split}
\end{equation*}
By Lemma~\ref{lema:qc}, there exists $\kappa=\kappa(g)\geq 1$ such
that
\[
g\inv C(\bar f_1,c\psi_1)\subset C(\pi(g^{-1}f_1),\kappa c \psi_1).
\]
Hence,
\[
\#(\Gamma x_0\cap C(\pi(g\inv f_1),\kappa c\psi_1))=\infty.
\]
This shows that $\pi(g\inv f_1)\in F$ for every $g\in\tilde E$, where
\begin{equation*}
  F=\{v\in \partial X:\, \exists \kappa\geq 1 \text{ such that }
  \#(\Gamma x_0\cap C(v,\kappa c\psi_1))=\infty\}.  
\end{equation*}
Since $\tilde E$ has positive measure, we conclude that $F$ has
positive measure. It follows from Lemma~\ref{lema:qc} that $F$ is
$\Gamma$-invariant.  Therefore, by Proposition~\ref{prop:Gamma}, $F$
has full measure.

For $v\in F$, there exists a sequence $x_n\in \Gamma x_0$ such that
$\norm{x_n}\to\infty$ and $x_n\in C(v,\kappa c\psi_1)$ for some
$\kappa\geq 1$, i.e.,
\[
\norm{\pi(x_n)-v}\leq \kappa c\psi_1(\norm{x_n}),\quad \forall n\in\N.
\]
Since $\frac{\psi(t)}{\psi_1(t)}\to \infty$ as $t\to\infty$, it
follows that for all sufficiently large $n$,
\[
\norm{\pi(x_n)-v}\leq \psi(\norm{x_n}).
\]
This shows that every element of $F$ is $(X,\psi)$-approximable and
completes the proof of the first part of Theorem~\ref{thm:main-qform}.
\end{proof}

\begin{proof}[Proof of Theorem \ref{thm:main-qform}(ii)]
  Suppose that
  \[
  \int_0^\infty(e^t\psi(e^t))^{d-1}\,dt<\infty.
  \]
  Then by Lemma~\ref{lema:phi}(i), $e^t\psi(e^t)\to 0$ as
  $t\to\infty$.  Let
  \[
  W=\{v\in\partial X:\, \Gamma x_0\cap C_T(v,\psi)
  \neq\emptyset\quad\forall T>1\}.
  \]
  Note that by the theorem of Borel and Harish-Chandra, the set
  $X(\mathbb{Z})$ is a union of finitely many $\Gamma$-orbits. Hence,
  the set of $(X,\psi)$-approximable points is a finite union of sets
  of the form $W$.  It remains to show that $W$ has measure zero.

  Let $\tilde W=\{g\in G:\, \pi(gf_1)\in W\}$ and $\tilde W_0$ a
  bounded subset of $\tilde W$.  By Lemma~\ref{lema:qc}, there exists
  $\kappa>1$ such that
  \[
  g\inv \Gamma x_0\cap C_T(\bar f_1,\kappa\psi)\ne \emptyset,\quad
  \forall g\in\tilde W_0,\;\;T>1,
  \]
  Then by Lemma~\ref{prop:CD}, there exists $c_1>1$ such that
  \begin{equation*}
    g\inv \Gamma\cap D_T(\bar f_1,c_1 \psi)\ne\emptyset, \quad \forall g\in\tilde W_0,\;\;T>1,
  \end{equation*}
  and by (\ref{eq:42}), there exists $c_2>1$ such that
  \begin{equation}
    \label{eq:Dat}
    g\inv \Gamma\cap \left(\bigcup_{t\geq T} a_tu(\cU_t(c_2\psi))w_0\right)\ne\emptyset, \quad \forall
    g\in\tilde W_0,\;\; T>1.
  \end{equation}
  Let
  \[
  B_T=\bigcup_{t\geq T} a_tu(\cU_t(c_2\psi))Hg_0\Gamma/\Gamma\subset
  G/\Gamma.
  \]
  Since $G\cong\hbox{O}(p,q)$, we have that $H\cong\hbox{O}(p-1,q)$ or
  $\hbox{O}(p,q-1)$. Moreover, because $d=p+q\geq 4$, $H$ is a
  semisimple group. Hence, $H_0=\Stab_G(x_0)=g_0^{-1}H_0g_0$ is a
  semisimple group defined over $\Q$, and the space
  $Hg_0\Gamma/\Gamma=g_0H_0\Gamma/\Gamma$ admits a finite
  $H$-invariant measure.  Then there exist constants
  $\kappa_1,\kappa_2>1$ such that
  \begin{equation*}
    \vol_{G/\Gamma}(B_T)\leq \kappa_1 \int_T^\infty \vol_{\mathbb{R}^{d-1}}(\tilde \cU_t(c_2\psi))\,dt\leq
    \kappa_2 \int_T^\infty (e^t\psi(e^t))^{d-1}<\infty.
  \end{equation*}
  Hence, $\vol_{G/\Gamma}(B_T)\to 0$ as $T\to\infty$.  By
  \eqref{eq:Dat}, $\tilde W_0^{-1} \Gamma\subset B_T$ for all $T>1$.
  Therefore, $\vol_{G/\Gamma}(\tilde W_0^{-1}\Gamma)=0$. This implies
  that $\vol_G(\tilde W_0)=\vol_G(\tilde W_0\inv)=0$.  And hence
  $\mu_\infty(\pi(\tilde W f_1))=\mu_\infty(W)=0$. This completes the
  proof of Theorem~\ref{thm:main-qform}.
\end{proof}

\section{Examples}\label{sec:examples}

We give examples to illustrate that the main theorem holds only on a
set of points of full measure, but not everywhere.

Let $Q$ be a positive definite quadratic form with rational
coefficients in $d$ variables.  In this section, we say that a vector
$v\in \{Q=1\}$ is $(\epsilon,s)$-approximable if there exist integer
solutions of
\begin{equation}\label{eq:s}
  \left\|v-\frac{x}{y}\right\|< \frac{\epsilon}{|y|^{s}}, \quad\quad  Q(x)-y^2=-1
\end{equation}
with $|y|\to\infty$.  When $d\ge 3$, our main theorem implies that for
$\epsilon>0$ and $s\in [0,1]$, almost every $v\in\{Q=1\}$ is
$(\epsilon,s)$-approximable, and for $\epsilon>0$ and $s>1$, almost
every $v\in\{Q=1\}$ is not $(\epsilon,s)$-approximable.

\begin{ex}
  Let $d\ge 4$. Then there exist $\epsilon>0$ and a vector
  $v\in\{Q=1\}$ which is not $(\epsilon,1)$-approximable.
\end{ex}

Since $d\ge 4$, by the Meyer theorem, there exists a rational vector
$v$ such that $Q(v)=1$.  Let $k\in\mathbb{N}$ be such that $kv\in
\mathbb{Z}^d$ and
$$\epsilon<\min\left\{\|z\|:\, z\in \frac{1}{k}\mathbb{Z}^d-\{0\}\right\}.$$
If for some $(x,y)\in \mathbb{Z}^{d+1}$,
$$
\|yv-x\|<\epsilon,
$$
then $x=yv$ and $Q(x)=y^2$. Hence, (\ref{eq:s}) fails.

\begin{ex}
  Let $d\ge 2$ and $Q(x)=\sum_{i=1}^d x_i^2$. Then there exists a
  vector $v\in \{Q=1\}$ which is $(\epsilon,2)$-approximable for
  $\epsilon>2/\sqrt{7}$.
\end{ex}

Let $v =(\sqrt{7}/4, 3/4,0,\ldots,0)$.  We claim that there are
infinitely many integer solutions of
\begin{equation}\label{eq:s2}
  y^2\sum_{i=1}^d (x_i-yv_i)^2<\epsilon^2,\quad\quad\sum_{i=1}^d (x_i^2-y^2v_i^2)=1.
\end{equation}
We use that there are infinitely many integer solutions $(k_j,l_j)$,
$k_j,l_j\to\infty$, of the Pell equation $k^2-7l^2=1$ and take
$$
x_1^{(j)}=k_j,\quad y^{(j)}=4l_j,\quad x_i^{(j)}=y^{(j)}v_i,\;\; i>1.
$$
Then $(x^{(j)},y^{(j)})$ satisfies the second condition in
(\ref{eq:s2}) and $x_1^{(j)}-y^{(j)}v_1\to 0$ as $j\to\infty$.  For
$\epsilon>(2v_1)^{-1}=2/\sqrt{7}$ and sufficient large $j$,
$$
y^{(j)}=(2v_1)^{-1}((x_1^{(j)}+v_1y^{(j)})- (x_1^{(j)}-v_1y^{(j)}))<
\epsilon (x_1^{(j)}+v_1y^{(j)}).
$$
Then $(x^{(j)},y^{(j)})$ satisfies the first condition in
(\ref{eq:s2}).  This shows that $v$ is $(\epsilon,2)$-approximable.

\end{document}